\documentclass[10pt,reqno]{amsart}
\usepackage{graphicx}
\usepackage{hyperref}
\usepackage{amssymb}
\usepackage{verbatim}
\usepackage{empheq} 
\usepackage{bbm}
\usepackage[usenames,dvipsnames]{color}

\usepackage{hyperref}
\hypersetup{
    colorlinks=true,
    linkcolor=blue,
    filecolor=magenta,      
    urlcolor=cyan,
    pdftitle={Sharelatex Example},
    bookmarks=true,
    pdfpagemode=FullScreen,
}

\newtheorem{theorem}{Theorem}[section]
\newtheorem{corollary}[theorem]{Corollary}
\newtheorem{lemma}[theorem]{Lemma}
\newtheorem{proposition}[theorem]{Proposition}

\newtheorem{thmx}{Theorem}

\newtheorem{definition}[theorem]{Definition}
\newtheorem{remark}[theorem]{Remark}

\numberwithin{equation}{section}

\def\ga{\alpha}

   \def\CK{{\mathcal K}}

\def\R{\mathbb{R}}

\let\vp=\varphi

\let\ol=\overline

\let\.=\cdot
\let\0=\emptyset

\def\O{\Omega}
\def\M{\mathcal{M}}

\def\R{\mathbb{R}}

\def\JJ{\mathcal{J}}

\def\MM{\mathcal{M}}
\def\XX{\mathcal{X}}

\def\lan{\langle}
\def\ran{\rangle}

\def\bs{\backslash}
\def\ol{\overline}
\def\sm{\setminus}

\def\wt{\widetilde}

\def\stst{\subset\subset}

\def\al{\alpha}

\def\ep{\epsilon}

\def\la{\lambda}

\def\si{\sigma}

\def\Om{\Omega}

\def\ga{\gamma}

\def\vp{\varphi}

\begin{document}


\title[Nonlocal dispersal operators of Neumann type]{Principal spectral theory of time-periodic nonlocal dispersal operators of Neumann type}

\thanks{The author was partially supported by Vietnam National Foundation for Science and Technology Development (NAFOSTED) under grant 101.02-2018.312.}

\author{Hoang-Hung Vo$^{*}$}
\address{Department of Mathematics and Applications, Saigon University, 273 An Duong Vuong st., Ward 3, Dist.5, Ho Chi Minh City, Viet Nam}
\email{vhhungkhtn@gmail.com}


\begin{abstract} 
In this communication, we prove some important limits of the principal eigenvalue for nonlocal  operator of Neumann type with respect to the parameters, which are significant in the understanding of dynamics of biological populations. We obtained a complete picture about limits of the principal eigenvalue in term of the large and small dispersal rate and dispersal range classified by "Ecological Stable Strategy" of persistence. This  solves some open problems  remainning in the series of work \cite{BCV1,ShXi15-2,ShenVo1}, in which we have to overcome the new difficulties comparing  to \cite{BCV1,ShXi15-2,ShenVo1} since principal eigenvalue of nonlocal Neumann operator is not monotone with respect to the domain.  The maximum principle for this type of operator   is also achieved in this paper.
\end{abstract}

\subjclass[2010]{Primary 35B50, 47G20; secondary 35J60}



\keywords{Nonlocal Neumann operator, principal eigenvalue, generalized principal eigenvalue, maximum principle}

\maketitle

\tableofcontents

\section{\bf Introduction}

The reaction-diffusion equation with nonlocal dispersal has become the subject of intensive research since the past decade not only because it is not only more mathematically challenging but it can also be used to describe many phenomena in the real world more precisely. In many biological systems, organisms can travel for some distance and the transition probability from one location to another usually depends upon the distance the organisms traveled. Such dispersal is referred to as nonlocal dispersal and is usually modeled by proper integral operators of Neumann type (see \cite{RDL,LHDGM}). In this paper, we are concerned with the following nonlocal dispersal operator of Neumann type

\begin{equation}\label{main-eqn-linear}
L[u](t,x):=-u_{t}(t,x)+\frac{D}{\si^{k}}\int_{\Om}J\left(\frac{x-y}{\si}\right)\frac{u(t,y)-u(t,x)}{\si^{N}}dy+a(t,x)u(t,x),\quad (t,x)\in\R\times\ol{\Om},
\end{equation}
where usually $D$ is the dispersal rate, $\sigma$ is referred the dispersal range and $k$ is the ecological stable
strategy (ESS). 

The nonlocal dispersal problem of Neumann type is the problem of intensive interest in the framework of reaction-diffusion equations since it has real applications in the natural science. Many  applications  have been studied in the nice works \cite{RDL,Hamilton,LHDGM,HL,HMKV}, which are modeled by the nonlocal dispersal of Neumann type equations. On the aspect of mathematical analysis, it is strongly linked to local reaction diffusion equations with Neumann boundary condition, which are established in the important works of Cortazar et al. \cite{Cor1}, Andrew et al. \cite{AMRT} for   parabolic operators, Ishii and Nakamura \cite{IN} for quasilinear elliptic operator. Moreover, another approximation of the spectrum of linear elliptic  by using  Galerkin--Fourier method has also been done by Andr\'es et al. \cite{AM}. The concept of ESS comes from games theory and goes back to the work of
Hamilton \cite{Hamilton} 1967 on the evolution of sex-ratio. Roughly speaking, an ecological stable
strategy is a strategy such that if most of the members of a population adopt it, there
is no €mutant strategy that would yield a higher reproductive fitness. In this framework
the strategies are compared using their relative pay-off.  This concept has been
recently used and adapted to investigate ecological stable strategies of dispersal in
several contexts: unconditional dispersal by Cosner and Lou \cite{CL},  Hambrock and Lou \cite{HL}, nonlocal dispersal by Berestycki et al. \cite{BCV2}, Hutson et al. \cite{HMKV}. The concept of ESS is also used to study the evolution of dispersal for biological species in the close/open advective environments in the series of work of Lam, Lou and Lutscher \cite{LLL,LL}. Before stating the main results, let us give the main assumptions in this paper.

Throughout this paper, we assume
\begin{itemize}
\item[\rm\bf(H1)] $\O\subset\R^{N}$ is a bounded and connected domain with smooth boundary and $D,\si>0$, $k\geq0$.

\item[\rm\bf(H2)] $J\in C(\R^{N})$ is nonnegative, continuous and  supported in $B_\gamma(0)$ for some $\gamma>0$, and satisfies $J(0)>0$ and $\int_{\R^{N}}J(x)dx=1$, where $B_\gamma(0)\subset\R^{N}$ is the open ball centered at $0$ with radius $\ga$. 

\item[\rm\bf(H3)] $a\in C_{T}(\R\times\ol{\Om})$ for some $T>0$, where
$$
C_{T}(\R\times\ol{\Om})=\left\{v\in C(\R\times\ol{\Om}):v(t+T,x)=v(t,x),\quad\forall (t,x)\in\R\times\ol{\Om}\right\}.
$$  
\end{itemize}

We  denote
$$ 
a_{T}(x)=\frac{1}{T}\int_{0}^{T}a(t,x)dx\quad\quad x\in\ol{\Om}.
$$

In this article, we focus on the following goals :

$\bullet$ We first study the effects of the dispersal rate and the dispersal range on $\la_{1}(-L)$. The study of asymptotic behaviors of the principal eigenvalue  plays an important role in the understanding the persistence of species in the inclement environments. For instance, as the dynamics of the population under the  phenomena of climate change, one can understand that small diffusion rate expresses the environment of the population is colder and large diffusion rate expresses the environment of the population is hotter.  As is known by Berestycki et al. \cite{BCV2} that the population modelled by Fisher-KPP type nonlinearity persists if and only if the principal eigenvalue of the linearized operator as (\ref{main-eqn-linear}) is strictly negative. Therefore, study the limits of the principal eigenvalue for small and large dispersal rate and the dispersal range is important to understand these phenomena.

$\bullet$ Second, we study the maximum principle for nonlocal operator of Neumann type, which is of independent interest and as an application of the principal eigenvalue of the operator.

To these aims, let us start with the definition of principal spectral point of $L$. For the sake of presentation, we define the  spaces $\XX$, $\XX^{+}$ and $\XX^{++}$  as follows:
\begin{equation*}
\begin{split}
\XX&=C^{1,0}(\R\times{\ol\Om})\cap C_{T}(\R\times\ol{\Om}),\\
\XX^{+}&=\left\{v\in\XX:v(t,x)\geq0,\,\,(t,x)\in\R\times{\ol\Om}\right\},\\
\XX^{++}&=\left\{v\in\XX:v(t,x)>0,\,\,(t,x)\in\R\times{\ol\Om}\right\},
\end{split}
\end{equation*}
where $C^{1,0}(\R\times\ol{\O})$ denotes the class of functions $C^1$ in $t$ and continuous in $x$.

\begin{definition}\label{eigenvalue}
The {\rm principal spectrum point} of $-L_{\O}$ is defined by
\begin{equation*}
\la_{1}(-L)=\inf\left\{\Re\la:\la\in\si(-L)\right\},
\end{equation*}
where $\si(-L)$ is the spectrum of $-L$. If $\la_{1}(-L)$ is an isolated eigenvalue of $-L$ with an eigenfunction in $\XX^{+}\bs\{0\}$, then it is called the {\rm principal eigenvalue} of $-L$.  
\end{definition}

To state the main results, we first recall the following results.

\begin{theorem}[\cite{RS12}]\label{prop-principal-e-Neumann}
Suppose {\bf(H1)}-{\bf(H3)}.
\begin{enumerate}
\item When $\la_{1}(-L)$ is a principal eigenvalue of $-L$, it is geometrically simple and has an eigenfunction in $\XX^{++}$.

\item If $\la\in\R$ is an eigenvalue of $-L$ associated with an eigenfunction in $\XX^{+}$, then 
$$
\la=\la_{1}(-L)<\la_{*}:=\min_{x\in\ol{\Om}}\left[\frac{D}{\si^{k}}\int_{\Om}J\left(\frac{x-y}{\si}\right)\frac{1}{\si^{N}}dy-a_{T}(x)\right]
$$ 
and $\la$ is the principal eigenvalue.

\item If $\la_{1}(-L)<\la_{*}$, then $\la_{1}(-L)$  is the principal eigenvalue of $-L$.
\end{enumerate}
\end{theorem}



The asymptotic behaviours of the principal eigenvalue with respect to the parameters are very important in the understanding the global dynamics of a biolgical species population in the inclement environments. The practical meaning of a such study is well explained in the celebrated work of Y. Lou \cite{Lou} and the recent work of H. Berestycki et al. \cite{BCV2}. It has also attracted a lot of attentions of  the community of reaction-diffusion equations \cite{HeNi,HMKV,LHDGM,LamLou,LLL,LL,ShXi15-1,ShXi15-2,SLW,YLR}. However, the qualitative properties of the principal eigenvalue of nonlocal operator of Neumann type are rather less. Up to our knowledge, this is the first time, the limits with respect to the diffusion rate and diffusion range classified by the ESS, of the principal eigenvalue of time periodic nonlocal Neumann operator have been investigated.

Our first main result is about the effect of the dispersal rate $D$ on the principal spectrum point $\la_{1}(-L)$. To highlight the dependence on $D$, we write $L_{D}$ for $L$. 

\begin{thmx}\label{thm-scaling-limit-eigenvalue1-Neumann}
Assume {\bf(H1)}-{\bf(H3)}. Suppose $J$ is symmetric with respect to each component. Then the following statement hold:
\begin{enumerate}
\item   $\la_1(-L_D)\leq -\frac{1}{|\Omega|T}\int_{0}^{T}\int_{\O}a(t,x)dxdt$.

\item  $\lim_{D\to0}\la_{1}(-L_D)=-\max_{\ol{\Om}}a_{T}$.

\item  $\lim_{D\to\infty}\la_{1}(-L_D)=-\frac{1}{T|\Om|}\int_{0}^{T}\int_{\Om}a(t,x)dxdt$.
\end{enumerate}
\end{thmx}

The next  result is about effect of the dispersal range characterized by $\si$    on the principal spectrum point $\la_{1}(-L)$ classified by ecological stable strategy. To highlight the dependence on $\si$ and $k$, we write $L_{\si,k}$ for $L$.

\begin{thmx}\label{limit-sigma}
Assume {\bf(H1)}-{\bf(H3)}. The following statements hold.  

\begin{enumerate}
\item For each $k\geq0$, there holds 
$$
\lim_{\si\to\infty}\la_{1}(-L_{\si,k})=-\max_{\ol{\Om}}a_{T}.
$$

\item  If $k=0$, $a(t,x)$ is Lipschitz continuous with respect to $x$, there holds 
\begin{equation*}
\lim_{\si\to0^{+}}\la_{1}(-L_{\si,0})=-\max_{\ol{\Om}} a_T.
\end{equation*}

\end{enumerate}
\end{thmx}
\begin{remark} Due to corollary Theorem B (3) in \cite{RS12}, one sees that $L_{\sigma,0}$ always admits a principal eigenvalue for  $0<\sigma\ll 1$. Therefore, it is interesting that there is no need to impose any additional condition on $a(t,x)$ so that the limit in (2) holds.

\end{remark}

We obtain a nice picture of the limits of principal spectrum point with respect to large and small dispersal range classified by the ecological stable
strategy. Theorem \ref{limit-sigma} (1) extends Theorem 2.3 (2) of Shen and Xie \cite{ShXi15-1} for the time-periodic operator, however, the lack of variational formular yields significant difficulties, for which we have to use a different technique to prove the limit. The main idea is to consider the relation between the generalized principal eigenvalues of the nonlocal operators of Neumann and Diriclet types, which was previously considered by Shen-Vo \cite{ShenVo1}. Comparing the notions of generalized principal eigenvalue of local operators with Dirichlet and Neumman boundary conditions (see for instance the point (1.2) in \cite{Patrizi}), it is worth to pointing out that in the definitions of the generalized principal eigenvalue, it is necessary to impose the boundary condition for Neumann eigenvalue while there is no constraint on the boundary for Dirichlet eigenvalue. Here, the main difference is that no boundary constraint for both definitions of generalized principal eigenvalue of Dirichlet and Neumann types (see  (\ref{characterization}) and Definition \ref{Neuman-type} below). To prove Theorem B (2), we must overcome additional difficulties due to the structure of the nonlocal Neumman operator is that the principal eigenvalue is not monotone with respect to the domain, therefore new technique has been employed to deal with this issue. 

Next, we study the maximum principle. It is well-known that, one of the most interesting properties of the principal eigenvalue for an elliptic or parabolic operator is its use to characterize the the validity of the maximum principle. The validity of the maximum principle for nonlocal elliptic operator has first been  characterized by Coville \cite{Co10}. Recently, in the previous work \cite{ShenVo1}, we also obtained the characterization for maximum principle for nonlocal parabolic operator of Dirichlet type as following

\begin{definition}[Maximum principle] \label{defn-mp}
We say that $L_\O$ admits the {\rm maximum principle} if for any function $u\in C^{1,0}([0,T]\times\ol\O)$ satisfying
 \begin{equation}\label{mp-assumption}
 \begin{cases} 
L_\O[u]\leq 0 & \text{in $(0,T]\times\O$},\\
 u\geq0 & \text{on $(0,T]\times\partial\O$},\\
u(0,\cdot)\geq u(T,\cdot) & \text{in $\Om$},
 \end{cases}
\end{equation}
there must hold $u>0$ in $[0,T]\times\O$ unless $u\equiv0$ in $[0,T]\times\O$.
\end{definition}

\begin{theorem}[Maximum principle]\label{thm-mp-introduction}
Suppose {\bf(H1)}. If $\la_{1}(-L_{\Om})$ is the principal eigenvalue,  then $L_{\Om}$ admits the maximum principle if and only if $\lambda_1(-L_\O)\geq0$.
\end{theorem}

We emphasize that there is an interesting difference between the characterization of maximum principle for local and nonlocal operator is that for local operators the maximum principle holds if and only if the principal eigenvalue is strictly positive while for the nonlocal operator, the eigenvalue is only needed to be nonnegative. The maximum principle is well-known to be one of the most important tools in analysis to prove the  well-posedness, the asymptotic behavior or even the symmetry of the solutions of elliptic or parabolic equations. The applications of the maximum principle for time-independent nonlocal operators involving continuous and fractional kernels can be found in the nice recent works \cite{HRSV,JW1,JW2}. Therefore,  obtaining the condition for the validity of maximum principle  plays an important role in the community of reaction-diffusion equation and related fields. Therefore, our task in the current paper is to obtain such a result for nonlocal operator of Neumann type, which is also of independent interest. More precisely, we prove

\begin{thmx}\label{thm-mp-1} 
Assume {\bf(H1)}-{\bf(H3)}. If $\la_{1}(-L)$ is the principal eigenvalue of $-L$, then the following conditions are equivalent.
\begin{enumerate}
\item $\lambda_1(-L)\geq0$

\item $L$ possesses a super-solution in $\XX^{++}$, namely, there exists $\vp\in \XX^{++}$ such that $L[\varphi](t,x)\leq0$ for $(t,x)\in\R\times\ol{\Om}$.

\item Any strict super-solution  $\varphi\in \XX$ of $L$ must be strictly positive. In other words, $L$ satisfies the strong maximum principle.
\end{enumerate}
\end{thmx}

\begin{thmx}\label{thm-mp-2}
Assume {\bf(H1)}-{\bf(H3)}. If $\la_{1}(-L)$ is the principal eigenvalue of $-L$, then the following conditions are equivalent:
\begin{enumerate}
\item $\lambda_1(-L)>0$

\item $L$ possesses a strict  super-solution in $\XX^{+}$, namely, there exists $\vp\in \XX^+$ such that $L[\varphi](t,x)<0$ for $(t,x)\in\R\times\ol{\Om}$.
\end{enumerate}
\end{thmx}

The rest of the paper is organized as follows : In Section 2, we recall and prove some preliminary results that are necessary for the proofs of main results. Section 3 is devoted to the proofs of Theorem A and Theorem B. In Section 4, we prove the maximum principle stated in Theorem C and Theorem D. In the Appendix, we prove an additional result for the limit of the eigenvalue with respect to the dispersal range for the case $a(t,x)=a(x)$ and $k>2$.


\section{\bf Preliminaries}

In this section, we establish some necessary tools for later use. Before that, let us recall the following result :

\begin{theorem}
[\textbf{Theorem A }  \cite{ShenVo1}\label{thm-pe-introduction}]
Suppose {\bf(H1)}  and let $ a\in C_{T}(\R\times\ol{\Om})$ and  $\wt\lambda_{p}(a,-\wt L_\O)$ is the principal spectrum point  of the operator $-\wt L_{\O}$, where $\wt L_{\O}$ is defined by
\begin{equation}\label{oper1}
\wt L_{\O}[\psi]=-\psi_{t}(t,x)+D\int_{{\O}}J_{}(x-y)\psi(t,y)dy-D\psi(t,x)+ a(t,x)\psi(t,x).
\end{equation}

\begin{enumerate}
\item  If 
\begin{equation}\label{newcond1}
\frac{1}{\max_{y\in\ol{\Om}}a_{T}(y)-a_{T}}\notin L^{1}_{loc}(\ol{\Om}),
\end{equation}
then $\wt\lambda_{p}(a,-\wt L_\O)$ is the principal eigenvalue of $-\wt L_{\Om}$. 

\item If  $\wt\lambda_{p}(a,-\wt L_\O)$ is the principal eigenvalue of $-\wt L_{\Om}$, then
$$
\la_{1}(-\wt L_{\Om})=\wt\lambda_{p}(a,-\wt L_\O)=\wt\lambda_{p}'(a,-\wt L_{\Om}),
$$
where $\la_{1}(-\wt L_{\Om})$ is defined in Definition \ref{eigenvalue} and
\begin{equation}\label{characterization}
\left\lbrace\begin{split}
\wt\lambda_{p}(a,-\wt L_\O):&=\sup\left\{\lambda\in\R:\exists\phi\in \mathcal{X}^{++}_{\Om}\,\,\text{s.t.}\,\,(\wt L_{\Om}+\la)[\phi]\leq0\,\,\text{in ${\R\times\ol\O}$}\right\},\\
\wt\lambda_{p}'(a,-\wt L_\O):&=\inf\left\{\lambda\in\R:\exists\phi\in \mathcal{X}^{++}_{\Om}\,\,\text{s.t.}\,\,(\wt L_{\Om}+\la)[\phi]\geq0\,\,\text{in ${\R\times\ol\O}$}\right\}.
\end{split}\right.
\end{equation}
\end{enumerate}\label{recal}
\end{theorem}

\begin{corollary}\label{corr2}  The non-integrability condition (\ref{newcond1}) is satisfied, which implies the operator (\ref{oper1}) admits a principal eigenpair, if $a_T(x)$ achieves a global maximum at some point $x_0\in\ol\O$ in the following cases :

i) $N=1$, $a_T(x)\in C(\O)$ 

ii) $N=2$, $a_T(x)\in C^1(\O)$ 

iii) $N\geq3$, $a_T(x)\in C^{N-1}(\O)$ and $\partial^k a_T(x_0)=0$ for all $k<N$.

\end{corollary}

Note that the condition \eqref{newcond1} concerns the smoothness of $a_{T}$ near its maximum points. { Moreover, it is independent of the dispersal kernel $J$ and the dispersal rate $D$, and hence, independent of the dispersal operator $u\mapsto D\left[\int_{\Om}J(\cdot-y)u(y)dy-u\right]$. Such a dispersal-independent sufficient condition is expected for the reason that $a(t,x)$ more or less determines the existence or non-existence of the principal eigenvalue under the current assumptions on $J$. This can be seen from the fact that the principal eigenvalue always exists when $a\equiv0$, which is implied by our sufficient condition and also a simple consequence of the facts that the operator $\mathcal{J}:u\mapsto D\int_{\Om}J(\cdot-y)u(y)dy$ on $C(\ol{\Om})$ is compact and $\JJ^{i}$ is strongly positive for some positive integer $i$. We further mention that the condition \eqref{newcond1} becomes very useful when we study scaling limits of the principal eigenvalue in the Theorem \ref{limit-sigma}.  Indeed, it allows us to prove a result on the uniform with respect to the dispersal range, approximation of the principal spectrum point, which says that $\la_{1}(-L_{\Om})$ is almost the principal eigenvalue and is of technical importance in the study of effects of large and small dispersal range on the principal eigenvalue under the lack of variational formula.

\subsection{Approximating the principal spectrum point}\label{subsec-approximation-PSP}

Define 
\begin{equation*}
\begin{split}
C^{+}(\ol{\Om})&=\left\{v\in C(\ol{\Om}):v(x)\geq0,\,\,x\in\ol{\Om}\right\},\\
C^{++}(\ol{\Om})&=\left\{v\in C(\ol{\Om}):v(x)>0,\,\,x\in\ol{\Om}\right\}. 
\end{split}
\end{equation*}
Denote by $\|\cdot\|_{\infty}$ the max norm on $C(\ol{\Om})$.  Consider the following linear equation
\begin{equation}\label{linear-eqn-08-30-18}
u_{t}(t,x)=\frac{D}{\si^{k}}\int_{\Om}J\left(\frac{x-y}{\si}\right)\frac{u(t,y)-u(t,x)}{\si^{N}}dy+a(t,x)u(t,x),\quad (t,x)\in\R\times\ol{\Om}.
\end{equation}
Denote by $\{\Phi(t,s)\}_{t\geq s\geq0}$ the evolution family of bounded linear operators on $C(\ol{\Om})$ generated by \eqref{linear-eqn-08-30-18}, that is, if $u(t,x;s,u_{0})$ is the unique solution of \eqref{linear-eqn-08-30-18} with initial data $u(s,\cdot;s,u_{0})=u_{0}\in C(\ol{\Om})$, then 
$$
u(t,\cdot;s,u_{0})=\Phi(t,s)u_{0}\in C(\ol{\Om}),\quad t\geq s. 
$$
By comparison principle, if $u_{0}\in C^{+}(\ol{\Om})$, so does $\Phi(t,s)u_{0}$ for all $t>s$. Moreover, if $u_{0}\in C^{+}(\ol{\Om})\bs\{0\}$, then $\Phi(t,s)u_{0}\in C^{++}(\ol{\Om})$ for all $t>s$. Also, by time-periodicity, one has 
$$
\Phi(t+T,s+T)=\Phi(t,s),\quad t\geq s\geq0. 
$$
The operator norm of $\Phi(t,s)$ is denoted by $\|\Phi(t,s)\|$. 

The next result connects $\Phi(t,s)$ with $\la_{1}(-L)$.

\begin{lemma}\label{lem-characterization-08-30-18}
Assume {\bf(H1)}-{\bf(H3)}. There hold
$$
-\la_{1}(-L)=\frac{\ln r(\Phi(T,0))}{T}=\limsup_{t-s\to\infty}\frac{\ln\|\Phi(t,s)\|}{t-s},
$$
where $r(\Phi(T,0))$ is the spectral radius of $\Phi(T,0)$.
\end{lemma}
\begin{proof}
See \cite[Proposition 3.3 and Proposition 3.10]{RS12}.
\end{proof}

The following results are scattered in \cite{RS12}. To highlight the dependence on $a(t,x)$, we write $L$ as $L(a)$, and $\Phi(t,s)$ as $\Phi(t,s;a)$.

\begin{proposition}\label{prop-approximation-08-31-18}
Assume {\bf(H1)}-{\bf(H3)}. For any $\ep>0$ and non-negative integers $p$ and $q$, there exists $a^{\ep}\in C_{T}(\R\times\ol{\Om})\cap C^{p,q}(\R\times\ol{\Om})$ such that the following hold:
\begin{enumerate}
\item $\max_{\R\times\ol{\Om}}|a^{\ep}-a|\leq\ep$;
\item $\la_{1}(-L(a^{\ep}))$ is the principal eigenvalue of $-L(a^{\ep})$;
\item there holds
$$
|\la_{1}(-L(a^{\ep}))-\la_{1}(-L(a)|\leq\ep
$$
for all $D>0$, $\si>0$ and $k\geq0$.
\end{enumerate}
\end{proposition}
\begin{proof}
By \cite[Lemma 4.1 and Theorem B (1)]{RS12}, for any $\ep>0$ there exists $\tilde{a}^{\ep}\in C_{T}(\R\times\ol{\Om})$ such that $\max_{\R\times\ol{\Om}}|\tilde{a}^{\ep}-a|\leq\frac{\ep}{2}$ and $\la_{1}(-L(\tilde{a}^{\ep}))$ is the principal eigenvalue of $-L(\tilde{a}^{\ep})$. The stability of isolated eigenvalues under bounded perturbation then allows us to find some $a^{\ep}\in C_{T}(\R\times\ol{\Om})\cap C^{p,q}(\R\times\ol{\Om})$ such that $\max_{\R\times\ol{\Om}}|a^{\ep}-\tilde{a}^{\ep}|\leq\frac{\ep}{2}$ and $\la_{1}(-L(a^{\ep}))$ is the principal eigenvalue of $-L(a^{\ep})$. This proves (1) and (2).

It remains to show (3). By the comparison principle, we find for any $u_{0}\in C^{+}(\ol{\Om})$
$$
\Phi(t,s;a^{\ep}-\ep)u_{0}\leq\Phi\left(t,s;\tilde{a}^{\ep}-\frac{\ep}{2}\right)u_{0}\leq \Phi(t,s;a)u_{0}\leq\Phi\left(t,s;\tilde{a}^{\ep}+\frac{\ep}{2}\right)u_{0} \leq \Phi(t,s;a^{\ep}+\ep)u_{0},\quad\forall t\geq s.
$$
As $\Phi(t,s;a^{\ep}\pm\ep)u_{0}=e^{\pm\ep(t-s)}\Phi(t,s;a^{\ep})u_{0}$, we find
$$
\frac{\ln\|\Phi(t,s;a^{\ep})\|}{t-s}-\ep\leq\frac{\ln\|\Phi(t,s;a)\|}{t-s}\leq\frac{\ln\|\Phi(t,s;a^{\ep})\|}{t-s}+\ep,\quad \forall t\geq s.
$$
The result then follows from Lemma \ref{lem-characterization-08-30-18}.
\end{proof}


\begin{remark}\label{rem-2019-03-29}{\rm
Note that the conclusions in Proposition \ref{prop-approximation-08-31-18} are uniform in $D>0$, $\si>0$ and $k\geq0$. This allows us to assume, without loss of generality, that the principal spectrum point $\la_{1}(-L)$ is the principal eigenvalue of $-L$ when studying the asymptotic behaviors of $\la_{1}(-L)$ as $D,\si\to0$ or $\infty$ for fixed $k\geq0$. Of course, this requires the asymptotic behaviors to depend on $a$ in a nice way, and this is not an issue in our case.}
\end{remark}

\subsection{Characterizations of the principal eigenvalue}

We prove the sup-inf characterizations of the principal eigenvalue, which together with approximation results in Subsection \ref{subsec-approximation-PSP} are very powerful tools in the investigation of the limits of the principal spectrum point with respect to the parameters.

\begin{definition}[Generalized principal eigenvalue]\label{Neuman-type}
The numbers
\begin{equation*}
\begin{split}
\lambda_{p}(-L):&=\sup\left\{\lambda\in\R:\exists\phi\in \mathcal{X}^{++}\,\,\text{s.t.}\,\,(L+\la)[\phi]\leq0\,\,\text{in ${\R\times\ol\O}$}\right\},\\
\lambda_{p}'(-L):&=\inf\left\{\lambda\in\R:\exists\phi\in \mathcal{X}^{++}\,\,\text{s.t.}\,\,(L+\la)[\phi]\geq0\,\,\text{in ${\R\times\ol\O}$}\right\}
\end{split}
\end{equation*}
are called {\rm generalized principal eigenvalues} of $-L$.

A pair $(\la,\phi)\in\R\times\XX^{++}$ is called a {\rm test pair} for $\la_{p}(-L)$ (resp. $\la_{p}'(-L)$) if $(L+\la)[\phi]\leq0$ in $\R\times\ol\O$ (resp. $(L+\la)[\phi]\geq0$ in $\R\times\ol\O$).
\end{definition}

\begin{theorem}\label{characterization}
Assume {\bf(H1)}-{\bf(H3)}. Suppose that $\la_{1}(-L)$ is the principal eigenvalue of $-L$, then
\begin{equation}\label{equivlence}
\la_{1}(-L)=\la_{p}(-L)=\lambda_p'(-L).
\end{equation}
\end{theorem}

\begin{proof}

Setting $\tilde{D}=\frac{D}{\si^{k}}$ and $\tilde{J}=\frac{1}{\si^{N}}J(\frac{\cdot}{\si})$, we may assume without loss of generality that $\si=1$. In this case,
$$
L[u](t,x):=-u_{t}(t,x)+D\int_{\Om}J(x-y)\left[u(t,y)-u(t,x)\right]dy+a(t,x)u(t,x),\quad (t,x)\in\R\times\ol{\Om}.
$$

For simplicity, we write $\la_{p}=\la_{p}(-L)$, $\la_{p}'=\la_{p}'(-L)$ and $\la_{1}=\la_{1}(-L)$.  We first prove 
$\la_{1}=\lambda_p$. By Definition \ref{eigenvalue} and Theorem \ref{prop-principal-e-Neumann} (1), there exists $\phi_1\in \XX^{++}$ such that 
\begin{equation}\label{pe-eqn-03-14}
L[\phi_{1}]+\la_{1}\phi_{1}=0\quad\text{in}\quad \R\times\ol{\Om}.
\end{equation}
Since $\inf_{\R\times\ol{\Om}}\phi_1>0$, there holds $\la_{1}\leq\la_{p}$ by the definition of $\la_{p}$. To show the equality, let us suppose for contradiction that $\la_{1}<\la_{p}$. From the definition of $\la_{p}$, we can find some $\la\in(\la_{1},\la_{p})$ and $\phi\in\XX^{++}$ such that 
\begin{equation}\label{pe-eqn-aux-03-14}
L+\la\phi\leq0\quad\text{in}\quad\R\times\ol{\Om}. 
\end{equation}
Clearly, $w:=\frac{\phi_{1}}{\phi}\in\XX^{++}$. Rewriting \eqref{pe-eqn-aux-03-14} as 
$$
-\phi_{t}+a(t,x)\phi\leq-\la\phi-D\int_{\Om}J(x-y)\left[\phi(t,y)-\phi(t,x)\right]dy, 
$$
we deduce
\begin{equation*}
\begin{split}
L[\phi_{1}]&=-w_{t}\phi+D\int_{\Om}J(x-y)\left[\phi(t,y)w(t,y)-\phi(t,x)w(t,x)\right]dy+[-\phi_{t}+a(t,x)\phi(x)]w\\
&\leq-w_{t}\phi+D\int_{\Om}J(x-y)\left[\phi(t,y)w(t,y)-\phi(t,x)w(t,x)\right]dy \\
&\quad+\left[-\la\phi-D\int_{\Om}J(x-y)\left[\phi(t,y)-\phi(t,x)\right]dy\right] w\\
&=-w_{t}\phi-\la\phi_{1}+D\int_{\Om}J(x-y)\phi(t,y)\left[w(t,y)-w(t,x)\right]dy.
\end{split}
\end{equation*}
It follows from \eqref{pe-eqn-03-14} that
\begin{equation}\label{an-inequality-000001}
-(\la_{1}-\la)\phi_{1}\leq-w_{t}\phi+D\int_{\Om}J(x-y)\phi(t,y)\left[w(t,y)-w(t,x)\right]dy.
\end{equation}

As $w\in\XX^{++}$, there exists $(t_{0},x_{0})\in\R\times\ol{\Om}$ such that $w(t_{0},x_{0})=\max_{\R\times\ol{\Om}}w$. Then, $w_{t}(t_{0},x_{0})=0$. Setting $(t,x)=(t_{0},x_{0})$ in \eqref{an-inequality-000001}, we find $-(\la_{1}-\la)\phi_{1}(t_{0},x_{0})\leq0$, which leads to $\la_{1}\geq\la$, a contradiction. This confirms $\la_{1}=\lambda_p$.

Next, we prove $\lambda_1=\lambda_p'$. Obviously $\lambda_1\geq\lambda_p'$. Assume that $\lambda_1>\lambda_p'$. There exist $\tilde{\lambda}\in (\lambda_p',\lambda_1)$ and $\tilde{\phi}\in\XX_{\Om}^{++}$ such that $L[\tilde{\phi}]+\tilde{\la}\tilde{\phi}\geq0$. Set $\tilde{w}:=\frac{\phi_1}{\tilde{\phi}}$. The same arguments as above apply and we derive
\begin{equation}\label{22.11.1}
0>-(\la_{1}-\tilde{\la})\phi_{1}\geq-\tilde{w}_{t}\phi+D\int_{\Om}J\left(x-y\right)\tilde{\phi}(t,y)\left[\tilde{w}(t,y)-\tilde{w}(t,x)\right]dy.
\end{equation}
We can find some $(t_1,x_1)\in \R\times\ol{\Om}$ such that $\tilde{w}(t_{1},x_{1})=\min_{\R\times\ol{\Om}}\tilde{w}$. Substituting $(t_1,x_1)$ into the right-hand side of \eqref{22.11.1}, we derive the contradiction.
\end{proof}


\section{{\bf Effects of parameters}}

In this section, we study the effects of the small and large dispersal rate $D$ and the dispersal range $\sigma$ on the principal spectrum point. In particular, we prove Theorem \ref{thm-scaling-limit-eigenvalue1-Neumann} and Theorem \ref{limit-sigma}.


\subsection{Effects of the dispersal rate}

In this subsection, we investigate the effects of the dispersal rate on the principal spectrum point and prove Theorem \ref{thm-scaling-limit-eigenvalue1-Neumann}. To highlight the dependence on $D$, we write $L$ as $L_{D}$.

We prove two lemmas before proving Theorem \ref{thm-scaling-limit-eigenvalue1-Neumann}. The first lemma gives some results on $\la_{1}(-L_{D})$ for small and large $D$.

\begin{lemma}\label{lem-boundedness-principal-eigenvalues}
Assume {\bf(H1)}-{\bf(H3)}. 
\begin{enumerate}
\item There hold
$$
-\max_{\R\times\ol{\Om}}a\leq \la_{1}(-L_{D})\leq-\min_{\R\times\ol{\Om}}a,\quad \forall D>0.
$$

\item For each $0<\ep\ll1$, there exists $D_{\ep}\in(0,1)$ such that
\begin{equation*}
-\max_{\ol{\Om}}a_{T}-\ep\leq\la_{1}(-L_D)\leq-\min_{\ol{\Om}}a_{T}+\ep,\quad\forall D\in(0,D_{\ep}).
\end{equation*}
\end{enumerate}
\end{lemma}
\begin{proof}
By Proposition \ref{prop-approximation-08-31-18} and an approximating argument (as explained in Remark \ref{rem-2019-03-29}), we may assume, without loss of generality, that $\la_{1}(-L_D)$ is the principal eigenvalue of $-L_D$. Then, $\la_{1}(-L_D)=\la_{p}(-L_D)=\la_{p}'(-L_{D})$ due to Theorem \ref{characterization}. 

(1) Let $\la=-\max_{\R\times\ol{\Om}}a$ and $\phi\equiv1$. It is easy to check that $(L_{D}+\la)[\phi]=a+\la\leq0$, namely, $(\la,\phi)$ is a test pair for $\la_{p}(-L_D)$, and hence, $\la_{1}(-L_{D})=\la_{p}(-L_D)\geq-\max_{\R\times\ol{\Om}}a$.

Similarly, it is easy to check that $(\la',\phi')=(-\min_{\R\times\ol{\Om}}a,1)$ is a test pair for $\la_{p}'(-L_{D})$, namely, $(L_{D}+\la')[\phi']\geq0$. It follows that $\la_{1}(-L_{D})=\la_{p}(-L_D)\leq-\min_{\R\times\ol{\Om}}a$. 

(2) It is easy to check that for each $x\in\ol{\Om}$, the function 
$$
t\mapsto\phi(t,x):=e^{\int_{0}^{t}\left[a(s,x)-a_{T}(x)\right]ds},\quad t\in\R
$$ 
is a positive $T$-periodic solution of $\phi_{t}=a(t,x)\phi-a_{T}(x)\phi$. In particular, $\phi\in\XX^{++}$.  For $0<\ep\ll1$, set 
$$
\la_{\ep}^{\max}=-\max_{\ol{\Om}}a_{T}-\ep\quad\text{and}\quad\la_{\ep}^{\min}=-\min_{\ol{\Om}}a_{T}+\ep.
$$ 
Using the fact $\min_{[0,T]\times\ol{\Om}}\phi>0$, it is straightforward to check that for each $0<\ep\ll1$, there exists $0<D_{\ep}\ll1$ such that for each $D\in(0,D_{\ep})$, there hold
$(L_{D}+\la_{\ep}^{\max})[\phi]\leq0$ and $(L_{D}+\la_{\ep}^{\min})[\phi]\geq0$. This together with the definitions of $\la_{p}(-L_{D})$ and $\la_{p}'(-L_{D})$ and Theorem  \ref{characterization} ensure that for each $0<\ep\ll1$, there holds $\la_{\ep}^{\max}\leq \la_{1}(-L_D)\leq\la_{\ep}^{\min}$ for all $D\in(0,D_{\ep})$. This completes the proof.
\end{proof}

In the second lemma, we prove a Poincar\'e-type inequality of the operator $\MM:L^{2}(\Om)\to L^{2}(\Om)$ defined by 
$$
\M[f](x)=-\int_\O J(x-y)[f(y)-f(x)]dy, \quad x\in\Om,
$$
where $J$ is as in {\bf(H2)}.

\begin{lemma}\label{lem-poincare-inequality}
Let $J$ be as in {\bf(H2)} and it is symmetric with respect to each component.  There holds
$$
\int_{\Om}\MM[f](x)f(x)dx=\frac{1}{2}\int_{\O}\int_\O J(x-y)\left[f(y)-f(x)\right]^2dydx,
$$
and there exists $C>0$ such that
$$
\int_{\Om}\MM[f](x)f(x)dx\geq C\int_{\Om}f(x)^{2}dx
$$
for all $ f\in L^{2}(\Om)$ with $\int_{\Om}f(x)dx=0$.
\end{lemma}
\begin{proof}
It is easy to check that $\MM$ is bounded and symmetric, hence, self-adjoint. Moreover,  for any $f\in L^{2}(\Om)$, we see from the symmetry of $J$ that
\begin{equation*}
\begin{split}
\lan\MM[f],f\ran_{L^{2}(\Om)}&=-\int_\O\int_\O J(x-y)f(y)f(x)dydx+\int_\O\int_\O J(x-y)f(x)^{2}dydx\\
&=-\int_\O\int_\O J(x-y)f(y)f(x)dydx+\frac{1}{2}\int_\O\int_\O J(x-y)f(x)^{2}dydx\\
&\quad+\frac{1}{2}\int_\O\int_\O J(x-y)f(y)^{2}dydx\\
&=\frac{1}{2}\int_{\O}\int_\O J(x-y)\left[f(y)-f(x)\right]^2dydx\geq0.
\end{split}
\end{equation*}
This says that $\MM$ is nonnegative. Then, there exists a unique bounded self-adjoint operator $\CK:L^{2}(\Om)\to L^{2}(\Om)$ such that $\MM=\CK^{2}$.

Obviously, $0$ is an eigenvalue of $\MM$ with constant functions on $\Om$ being eigenfunctions. Moreover, it is known from (see e.g \cite{ShXi15-1})  that $0$ is an isolated algebraically simple eigenvalue of $\M$. Thus, there holds the decomposition $L^2(\O)=E_1\oplus E_2$, where $E_1=\textrm{span$\{f\equiv1\}$}$ and $E_2=\{f\in L^2(\O):\int_\O fdx=0\}$. Moreover, $\M$ is invertible on $E_2$, which leads to the invertibility of $\CK$ on $E_{2}$. As a result, there is $C>0$ such that $\int_\O (\CK[f])^2dx\geq C\int_\O f^2dx$ for all $f\in E_{2}$. Since $\lan\MM[f],f\ran_{L^{2}}=\|\CK[f]\|_{L^{2}(\Om)}$, the lemma follows.
\end{proof}


Now, we prove Theorem \ref{thm-scaling-limit-eigenvalue1-Neumann}.

\begin{proof}[{\bf Proof of Theorem \ref{thm-scaling-limit-eigenvalue1-Neumann}}]
By Proposition \ref{prop-approximation-08-31-18} and an approximating argument, we may assume, without loss of generality, that $\la_{1}(-L_D)$ is the principal eigenvalue of $-L_D$. Moreover, we may assume, without loss of generality, that $\si=1$. Write $\la_{D}$ for $\la_{1}(-L_D)$ for simplicity. Let $\phi_{D}\in \XX^{++}$ satisfy the normalization $\|\phi_{D}\|_{L^2([0,T]\times\O)}=1$ and the eigen-equation
\begin{equation}\label{eigen_e}
L_{D}[\phi_{D}](t,x)+\la_{D}\phi_{D}(t,x)=0, \quad(t,x)\in\R\times\ol{\O}.
\end{equation}

(1) Dividing \eqref{eigen_e} by $\phi_D$ and integrating the resulting equation over $[0,T]\times\O$, we use the periodicity of $\phi_{D}$ to find
\begin{equation}\label{equality-2019-04-11}
D\int_{0}^{T}\int_{\Omega}\int_{\O}J(x-y)\left[\frac{\phi_{D}(t,y)}{\phi_{D}(t,x)}-1\right]dydxdt+\int_{0}^{T}\int_{\Omega}a(t,x)dxdt+T|\Omega|\lambda_{D}=0.
\end{equation}
The symmetry of $J$ ensures 
$$
\int_{\Om}\int_{\O}J\left( x-y\right)\left[\frac{\phi_{D}(t,y)}{\phi_{D}(t,x)}-1\right]dydx =\int_{\Om}\int_{\O}J(x-y)\left[\frac{\phi_{D}(t,x)}{\phi_{D}(t,y)}-1\right]dydx,\quad\forall t\in[0,T],
$$
which together with \eqref{equality-2019-04-11} yields
$$
\frac{D}{2}\int_{0}^T\int_{ \Omega}\int_{\O}J(x-y)\left[\sqrt{\frac{\phi_{D}(t,y)}{\phi_{D}(t,x)}}-\sqrt{\frac{\phi_{D}(t,x)}{\phi_{D}(t,y)}}\right]^2dydxdt+\int_{0}^{T}\int_{\Omega}a(t,x)dxdt+T|\Omega|\la_{D}=0.
$$
It follows that 
$$
\la_{D}\leq-\frac{1}{T|\Omega|}\int_{0}^{T}\int_{\Omega}a(t,x)dxdt.
$$

(2) By Theorem \ref{prop-principal-e-Neumann} and Lemma \ref{lem-boundedness-principal-eigenvalues} (2), for each $0<\ep\ll1$ there exists $D_{\ep}\in(0,1)$ such that
$$
-\max_{x\in\ol{\Om}}a_{T}(x)-\ep\leq\la_{D}\leq\min_{x\in\ol{\Om}}\left[D\int_{\Om}J\left(x-y\right) dy-a_{T}(x)\right],\quad\forall D\in(0,D_{\ep}).
$$
Letting $D\to0^{+}$, we find
$$
-\max_{x\in\ol{\Om}}a_{T}(x)-\ep\leq\liminf_{D\to0^{+}}\la_{D}\leq\limsup_{D\to0^{+}}\la_{D}\leq-\max_{x\in\ol{\Om}}a_{T}(x),\quad\forall0<\ep\ll1.
$$
The result follows.

(3) Multiplying \eqref{eigen_e} by $\phi_{D}$ and integrating the resulting equation over $[0,T]\times\O$, we find from the periodicity of $\phi_{D}$ and the normalization that
\begin{equation}\label{an-equality-090}
D\int_{0}^{T}\int_{\O}\left\{\int_\O J\left(x-y\right)\left[\phi_D(t,y)-\phi_D(t,x)\right]dy\right\}\phi_{D}(t,x)dxdt+\int_{0}^{T}\int_{\O}a(t,x)\phi_D^2(t,x)dxdt+\lambda_D=0.
\end{equation}
Calculations as in the proof of Lemma \ref{lem-poincare-inequality} using the symmetry of $J$ give
\begin{equation*}
\begin{split}
&\int_{0}^{T}\left\{\int_{\Om}\int_\O J\left(x-y\right)\left[\phi_D(t,y)-\phi_D(t,x)\right]^2dydx\right\}dt\\
&\quad\quad=-2\int_{0}^{T}\int_{\O}\left\{\int_\O J(x-y)\left[\phi_D(t,y)-\phi_D(t,x)\right]dy\right\}\phi_{D}(t,x)dxdt.
\end{split}
\end{equation*}
It then follows from \eqref{an-equality-090} that
\begin{equation*}
-\frac{D}{2}\int_{0}^{T}\int_{\O}\left\{\int_\O J(x-y)\left[\phi_D(t,y)-\phi_D(t,x)\right]^2dy\right\}dxdt+\int_{0}^{T}\int_{\O}a(t,x)\phi_D^2(t,x)dxdt+\lambda_D=0.
\end{equation*}
Since $a(t,x)$ is bounded and $\{\lambda_D\}_{D\gg1}$ is bounded by Lemma \ref{lem-boundedness-principal-eigenvalues} (1), there exists $C>0$ such that
\begin{equation}\label{estimate-08-07-18}
-\frac{D}{2}\int_{0}^{T}\int_{\O}\left\{\int_\O J(x-y)\left[\phi_D(t,y)-\phi_D(t,x)\right]^2dy\right\}dxdt\geq -C. 
\end{equation}

Define $\ol\phi_{ D}(t)=\frac{1}{|\O|}\int_\O\phi_D(t,x)dx$ for $t\in\R$ and set $\psi_{ D}=\phi_{ D}-\ol\phi_{D}$. Applying Lemma \ref{lem-poincare-inequality}, we find from \eqref{estimate-08-07-18} that
\begin{equation*}
\begin{split}
\int_{0}^{T}\int_{\Om}\MM[\psi_{D}(t,\cdot)](x)\psi_{D}(t,x)dxdt&=\frac{1}{2}\int_{0}^{T}\left\{\int_{\Om}\int_\O J(x-y)\left[\psi_D(t,y)-\psi_D(t,x)\right]^2dydx\right\}dt\\
&=\frac{1}{2}\int_{0}^{T}\left\{\int_{\Om}\int_\O J(x-y)\left[\phi_D(t,y)-\phi_D(t,x)\right]^2dydx\right\}dt\leq \frac{C}{D}.
\end{split}
\end{equation*}
Since $\int_\O\psi_D(t,x)dx=0$ for all $t\in[0,T]$, we can apply Lemma \ref{lem-poincare-inequality} to find 
$$
\int_{\Om}\psi_{D}^{2}(t,x)dx\leq C_{1}\int_{\Om}\MM[\psi_{D}(t,\cdot)](x)\psi_{D}(t,x)dx,\quad \forall t\in[0,T]
$$
for some $C_{1}>0$. Hence,

%

\begin{equation}\label{09.02.2}
\int_{0}^{T}\int_{\Om} \psi_{D}^2(t,x)dxdt\leq\frac{C_{1}C}{D}.
\end{equation}

Integrating \eqref{eigen_e} over $\O$ and dividing the resulting equation by $|\Om|$,  we find
\begin{equation*}
\begin{split}
\partial_t\ol\phi_D&=\frac{1}{|\O|}\int_\O\left\{D\int_\O J(x-y)\left[\phi_D(t,y)-\phi_D(t,x)\right]dy+a(t,x)\phi_D(t,x)+\lambda_D\phi_D(t,x)\right\}dx\\
&= \lambda_D\ol\phi_D+\frac{1}{|\O|}\int_\O a(t,x)\phi_D(t,x) dx.
\end{split}
\end{equation*}
Setting $\ol{a}(t)=\frac{1}{|\Om|}\int_{\Om}a(t,x)dx$, we find
\begin{equation*}
\partial_t\ol\phi_D-\left[\ol a(t)+\lambda_D\right]\ol\phi_D= \frac{1}{|\O|}\int_\O a(t,x)\left[\phi_D(t,x)- \ol\phi_D(t)\right]dx= \frac{1}{|\O|}\int_\O a(t,x)\psi_{D}(t,x)dx.
\end{equation*}
It follows from the variation of constants formula that
\begin{eqnarray*}
\ol\phi_D(t)=\ol\phi_D(0)e^{\int_0^t \left[\ol a(s)+\lambda_D\right]ds}+ \frac{1}{|\O|}\int_0^t e^{\int_{\tau}^{t} \left[\ol a(s)+\lambda_D\right]ds}\int_\O a(\tau,x)\psi_{D}(\tau,x)dxd\tau,\quad t\geq0.
\end{eqnarray*}
Since $a(t,x)$ and $\{\lambda_D\}_{D\gg1}$ are bounded, we deduce from \eqref{09.02.2} and H\"older's inequality that 
\begin{equation}\label{estimate-01-24-18}
\ol\phi_D(t)= \ol\phi_D(0)e^{\int_0^t \left[\ol a(s)+\lambda_D\right]ds}+O\left(\frac{1}{\sqrt{D}}\right),\quad \forall t\in[0,T]
\end{equation}
for all $D\gg1$. Since $\ol\phi_D(T)=\ol\phi_D(0)$, there must hold either $\ol\phi_D(0)\to0$ or $\int_0^T \left[\ol c(t)+\lambda_D\right]dt\to 0$ as $D\to\infty$. If $\ol\phi_D(0)\to0$ as $D\to\infty$, then \eqref{estimate-01-24-18} implies that $\ol\phi_D(t)\to0$ as $D\to\infty$ uniformly in $t\in[0,T]$. This together with \eqref{09.02.2} yields that $\phi_{ D}=\psi_{ D}+\ol\phi_{ D}$ converges in $L^2([0,T]\times\O)$ to $0$ as $D\to\infty$. However, $\|\phi_{D}\|_{L^2([0,T]\times\O)} =1$ for all $D\gg1$, which leads to a contradiction. Hence, there must hold $\int_0^T \left[\ol a(t)+\lambda_D\right]dt\to 0$ as $D\to\infty$, that is, 
$$
\lim_{D\to\infty}\lambda_D=-\frac{1}{T}\int_{0}^{T}\ol{a}(t)dt=-\frac{1}{T|\O|}\int_{0}^{T}\int_{\O} a(t,x)dxdt. 
$$
This completes the proof.
\end{proof}


\subsection{Effects of the dispersal range}

We study the effects of the dispersal range characterized by $\si$ on the principal spectrum point. To highlight the dependence on $\si>0$ and $k\geq0$, we write $L_{\si,k}$ for $L$. We prove Theorem \ref{limit-sigma}. 

\begin{proof}[\bf{Proof of Theorem \ref{limit-sigma}}]
By Proposition \ref{prop-approximation-08-31-18} and an approximating argument, we may assume, without loss of generality, that $\la_{1}(-L_{\si,k})$ is the principal eigenvalue of $-L_{\si,k}$.

(1) By Theorem \ref{prop-principal-e-Neumann} (2), we find 
$$
\la_{1}(-L_{\si,k})<\min_{x\in\ol{\Om}}\left[\frac{D}{\si^{k}}\int_{\Om}J\left(\frac{x-y}{\si}\right)\frac{1}{\si^{N}}dy-a_{T}(x)\right], 
$$
which implies 
$$
\limsup_{\si\to\infty}\la_{1}(-L_{\si,k})\leq-\max_{\ol{\Om}}a_{T}.
$$
It remains to show that 
\begin{equation}\label{to-prove-03-11-17}
\liminf_{\si\to\infty}\la_{1}(-L_{\si,k})\geq-\max_{\ol{\Om}}a_{T}.
\end{equation}
To do so, let us fix some constant $\phi_{0}>0$. It is easy to check that for each $x\in\ol{\Om}$, the function 
\begin{equation}
t\mapsto\phi(t,x):=e^{\int_{0}^{t}[a(s,x)-a_{T}(x)]ds}\phi_{0},\quad t\in\R\label{sollinear}
\end{equation}
is a positive $T$-periodic solution of the ODE $v_{t}=a(t,x)v-a_{T}(x)v$. Clearly, $\phi\in\XX_{\Om}^{++}$ and we may choose $\phi_{0}$ such that $\sup_{\R\times\ol\O}\phi=1$. For any $\delta>0$, we see that for each $(t,x)\in\R\times\ol{\Om}$,
\begin{equation}
\begin{split}
&\left(L_{\si,k}-\max_{\ol\Om}a_{T}-\delta\right)[\phi](t,x)\\
&\quad\quad=-\phi_{t}(t,x)+\frac{D}{\si^{k}}\left[\int_{\Om}J\left(\frac{x-y}{\si}\right)\frac{\phi(t,y)-\phi(t,x)}{\si^{N}}dy\right]+\left[a(t,x)-\max_{\ol\Om}a_{T}-\delta\right]\phi(t,x)\\
&\quad\quad\leq\frac{D}{\si^{k}}\left[\int_{\Om}J\left(\frac{x-y}{\si}\right)\frac{\phi(t,y)-\phi(t,x)}{\si^{N}}dy\right]-\delta\phi(t,x).
\end{split}\label{06.02.1}
\end{equation}
As $\min_{\R\times\ol{\Om}}\phi>0$ and $\left\|\frac{D}{\si^{k}}\int_{\Om}J\left(\frac{\cdot-y}{\si}\right)\frac{\phi(t,y)-\phi(t,\cdot)}{\si^{N}}dy\right\|_\infty\to0$ as $\si\to\infty$,
there is $\si_\delta>0$ such that 
$$
\left(L_{\si,k}-\displaystyle\max_{\ol\Om}a_{T}-\delta\right)[\phi]\leq0,\quad\forall\si\geq\si_\delta, 
$$
which implies that 
$$
\la_{1}(-L_{\si,k})=\la_{p}(-L_{\si,k})\geq-\max_{\ol\Om}a_{T}-\delta,\quad\forall \si\geq\si_\delta.
$$
The arbitrariness of $\delta>0$ then yields \eqref{to-prove-03-11-17}. Hence, the limit $\lim_{\si\to\infty}\la_{1}(-L_{\si,k})=-\max_{\ol{\Om}}a_{T}$ follows.

(2) For $k\in[0,1)$ and $x\mapsto a(t,x)$ is Lipschitz continuous, we first prove the inequality 
$$\liminf_{\si\to0^{+}}\la_{1}(-L_{\si,k})=\liminf_{\si\to0^{+}}\la_{p}(-L_{\si,k})\geq-\max_{\ol\Om}a_{T}.$$

 Let  $\ol\phi(t,x):=e^{\int_{0}^{t}\left[a(s,x)-a_{T}(x)\right]ds}$. Clearly, $x\mapsto\ol{\phi}(t,x)$ is Lipschitz continuous uniform in $t\in\R$, that is, there is $M>0$ such that
\begin{equation}\label{lip-cond-09-07-18}
\sup_{t\in\R}\left|\ol{\phi}(t,x)-\ol{\phi}(t,y)\right|\leq M|x-y|,\quad\forall x,y\in\ol{\Om}.
\end{equation}
For any $\ep>0$ and $(t,x)\in\R\times\ol{\Om}$, we have
\begin{equation*} 
\begin{split}
&\left(L_{\si,k}-\max_{\ol\Om}a_{T}-\ep\right)[\ol\phi](t,x) \\
&\quad\quad\leq \frac{D}{\si^{k}}\left\{\int_{\Om}J_{\si}(x-y)\left[\ol\phi(t,y)-\ol\phi(t,x)\right]dy\right\}-\ep\ol\phi(t,x)\\&\quad\quad=\frac{D}{\si^{k}}\left\{\int_{\Om}\frac{1}{\sigma^N}J\left(\frac{x-y}{\sigma}\right)\left[\ol\phi(t,y)-\ol\phi(t,x)\right]dy\right\}-\ep\ol\phi(t,x)\quad \quad \\
&\quad\quad=\frac{D}{\si^{k}}\int_{\frac{\O-x}{\sigma}}J(z)\left[\ol\phi(t,x+\sigma z)-\ol\phi(t,x)\right]dz-\ep\ol\phi(t,x)
\end{split}
\end{equation*}

By \eqref{lip-cond-09-07-18}, there holds
\begin{equation}
\left|\ol\phi(t,x+\si z)-\ol\phi(t,x)\right|\leq\si M|z|,\quad\forall x\in\ol{\Om},\quad z\in\frac{\Om-x}{\si}.\label{08.07.1}
\end{equation}
Hence,
\begin{equation*}
\begin{split}
\left(L_{\si,k}-\max_{\ol\Om}a_{T}-\ep\right)[\ol\phi](t,x)&\leq DM\si^{1-k}\int_{\frac{\Om-x}{\si}}J(z)|z|dz-\ep\ol\phi(t,x)\\
&<0, \quad \quad(t,x)\in\R\times\ol{\Om}
\end{split}
\end{equation*}
for all $0<\si\ll1$. The supremum characterization of $\lambda_p(-L_{\si,k})$ yields $\la_{1}(-L_{\si,k})\geq -\max_{\ol\Om}a_{T}-\epsilon$. Hence,
\begin{equation}\nonumber
\liminf_{\si\to0^{+}}\la_{1}(-L_{\si,k})=\liminf_{\si\to0^{+}}\la_{p}(-L_{\si,k})\geq-\max_{\ol\Om}a_{T}-\epsilon.
\end{equation}
The arbitrariness of $\epsilon$ implies
\begin{equation}\label{17.5.1}
\liminf_{\si\to0^{+}}\la_{1}(-L_{\si,k})=\liminf_{\si\to0^{+}}\la_{p}(-L_{\si,k})\geq-\max_{\ol\Om}a_{T}.
\end{equation}

Now let $k=0$, we shall prove the reverse inequality
\begin{equation}\label{17.5.1-03-12-17}
\limsup_{\si\to0^{+}}\la_{1}(-L_{\si})\leq-\max_{\ol\Om}a_{T}, 
\end{equation}
where
$$
 L_{\sigma}[\psi]=-\psi_{t}(t,x)+D\int_{{\O}}J_{\si}(x-y)(\psi(t,y)-\psi(t,x)) dy+ a(t,x)\psi(t,x).
$$ 

For any $\ep>0$, there exists an open ball of radius $\ep$ $B_\epsilon$  such that 
$a_T+\epsilon>\max_{\ol\O} a_T$ in $B_\epsilon\cap \O$. Let $\phi$ be a function defined as (\ref{sollinear}) such that $\sup_{\R\times\O}\phi=1$  and $\widetilde{\phi}_{\ep}:\R\times\R^{N}\to[0,\infty)$ be a $T$-periodic, continuous  function satisfying
\begin{equation*}
\widetilde{\phi}_{\ep}=\phi\,\,\text{in}\,\,\R\times\ol B_{\epsilon},\quad \widetilde{\phi}_{\ep}=0\,\,\text{in}\,\,\R\times(\R^{N}\bs B_{2\epsilon})\quad\text{and}\quad \sup_{\R\times\R^N}\widetilde{\phi}_{\ep}\leq \sup_{\R\times\R^N}\phi=1 .
\end{equation*}
Obviously, $\widetilde{\phi}_\epsilon(t,\cdot)\in C^4(\R^{N})$ for each $t\in\R$. In fact, we can assume, the approximation argument, $\phi$ is $C^{4}$, then $\widetilde{\phi}_\epsilon(t,\cdot)\in C^4(\R^{N})$ by its definition. We see
\begin{equation*}
\begin{split}
\mathcal{J}_{\sigma}(t,x):&=\int_{\R^{N}}J_{\si}(x-y)\left[\widetilde{\phi}_\epsilon(t,y)-\widetilde{\phi}_\epsilon(t,x)\right]dy\\
&=\int_{\R^N}J(z)\left[\widetilde\phi_\epsilon(t,x+\sigma z)-\widetilde{\phi}_\epsilon(t,x)\right]dz,
\end{split}
\end{equation*}
where the symmetry of $J$ with respect to each its component is used. By the fourth-order Taylor's expansion with remainder, we find
$$
\widetilde{\phi}_\epsilon(t,x+\si z)-\widetilde{\phi}_\epsilon(t,x)=\sum_{1\leq|\al|\leq 3}\frac{\partial^{\al}\widetilde{\phi}_\epsilon(t,x)}{\al!} \si^{|\al|}z^{\al}+\si^{4}\sum_{|\al|=4}R_{\al}(t,x)z^{\al},
$$
where $\al=(\al_{1},\dots,\al_{N})$ is the usual multiple index, and 
$$
R_{\al}(t,x)=\frac{4}{\al!}\int_{0}^{1}(1-s)^{3}\partial^{\al}\widetilde{\phi}_\epsilon(t,x+s\si z)ds.
$$
Since $J$ is symmetric with respect to each component, there hold $\int_{\R^N}J(z)z^{\al}dz=0$ for $|\al|=1$ or $3$ and $\int_{\R^N}J(z)z_{i}z_{j}dz=0$ for $i\neq j$. Therefore, 
\begin{equation*}
\mathcal{J}_{\sigma}(t,x)=\si^{2}\sum_{i=1}^{N}\frac{\partial_{x_{i}}^{2}\widetilde{\phi}_\epsilon(t,x)}{2} \int_{\R^{N}}J(z)z^{2}_{i}dz+\si^{4}\sum_{|\al|=4}R_{\al}(t,x)\int_{\R^{N}}J(z)z^{\al}dz.
\end{equation*}


Let $\wt\lambda_{p}(a,-\wt L^\sigma_\mathcal{O})$ be the principal eigenvalue of the operator $-\wt L^\sigma_\mathcal{O}$, where $\wt L^\sigma_\mathcal{O}$ is defined by
$$
\wt L^\sigma_\mathcal{O}[\psi]=-\psi_{t}(t,x)+D\int_{\mathcal{O}}J_{\si}(x-y)\psi(t,y)dy-D\psi+ a(t,x)\psi(t,x).
$$ 
Note that it is an operator of Dirichlet type.

Take $\ol\phi_\epsilon=\wt\phi_{\ep}=\phi$ in $\R^N\times \ol B_\epsilon$. For $(t,x)\in \R\times\ol B_\epsilon$, one has 
\begin{equation*}
\begin{split}
&\left(\wt L^\sigma_{B_\epsilon}-\max_{\ol{\Om}} a_{T}+\ep+\epsilon^{1/4}\right)[\ol\phi_\epsilon](t,x)\\
&\quad\quad=D\int_{B_{\ep}}J_{\si}(x-y)\ol\phi_\epsilon(t,y)dy-D\ol\phi_\epsilon(t,x)+\left[ a_T(x)-\max_{\ol{\Om}} a_{T}+\ep+\epsilon^{1/4}\right]\ol\phi_\epsilon(t,x)\\
&\quad\quad\geq D\left[\int_{B_\epsilon}J_{\si}(x-y)\wt\phi_\epsilon(t,y)dy-\wt\phi_\epsilon(t,x)\right]+\epsilon^{1/4}\wt\phi_\epsilon(t,x)\\
&\quad\quad=D\left[\int_{\R^N}J_{\si}(x-y)\wt\phi_\epsilon(t,y)dy-\wt\phi_\epsilon(t,x)-\int_{B_{2\epsilon}\setminus B_\epsilon}J_{\si}(x-y)\wt\phi_\epsilon(t,y)dy\right]+\epsilon^{1/4}\wt\phi_\epsilon(t,x)\\
&\quad\quad=D\mathcal{J}_{\sigma}(t,x)-D\int_{B_{2\epsilon}\setminus B_\epsilon}J_{\si}(x-y)\wt\phi_\epsilon(t,y)dy+\epsilon^{1/4}\wt\phi_\epsilon(t,x)\\
&\quad\quad=\mathcal{J}_{\sigma}^1(t,x)+\mathcal{J}_{\sigma}^2(t,x)+\mathcal{J}_{\sigma}^3(t,x) +\mathcal{J}_{\sigma}^4(t,x)
\end{split}
\end{equation*}
where
\begin{equation*}
\begin{split}
\mathcal{J}_{\sigma}^1(t,x)&=D\si^{2}\sum_{i=1}^{N}\frac{\partial_{x_{i}}^{2}\tilde{\phi}_\epsilon(t,x)}{2} \int_{\R^{N}}J(z)z^{2}_{i}dz,\\
\mathcal{J}_{\sigma}^2(t,x)&=-\frac{D}{\sigma^{N}}\int_{B_{2\epsilon}\setminus B_\epsilon}J\left(\frac{x-y}{\si}\right)\widetilde{\phi}_\epsilon(t,y)dy,\\
\mathcal{J}_{\sigma}^3(t,x)&=\epsilon^{1/4}\wt\phi_\epsilon(t,x)\quad\text{and}\\
\mathcal{J}_{\sigma}^4(t,x)&=D\si^{4}\sum_{|\al|=4}R_{\al}(t,x)\int_{\R^{N}}J(z)z^{\al}dz.
\end{split}
\end{equation*}
Since $\min_{\R\times\ol{\Om}}\phi>0$, one has $\min_{\R\times\ol B_\epsilon}\phi(t,x)>0$ uniformly in $\epsilon$. Choosing $\epsilon=\sigma^2$,  we find the following estimates hold 
$$
\sup_{\R\times\ol B_\epsilon}|\mathcal{J}_{\sigma}^1|\leq C_1\sigma^{2};\quad\sup_{\R\times\ol B_\epsilon}|\mathcal{J}_{\sigma}^2|\leq C_2\sigma^{N};\quad \inf_{\R\times\ol B_\epsilon}|\mathcal{J}_{\sigma}^3|\geq \sqrt{\sigma} C_3;\quad \sup_{\R\times\ol B_\epsilon}|\mathcal{J}_{\sigma}^4|\leq C_4\sigma^{4}.
$$
Indeed, the first, the third and the fourth ones are simple consequences of the fact that $\wt{\phi}_{\ep}=\phi$ on $\ol{B}_{\ep}$.  For the second one, it follows from the boundedness of $J$, $\wt\phi_\epsilon$ and the formula of a N-dimensional volume of a Euclidean ball of radius r 
$$V_N(r)=\frac{\pi^{N/2}}{\Gamma(\frac{N}{2}+1)}r^N,$$
where $\Gamma$ is the gamma function defined by  $\Gamma(N+\frac{1}{2})=(N-\frac{1}{2})(N-\frac{3}{2})...\frac{1}{2}.\pi^{\frac{1}{2}}.$

Since $N\geq1$, the term $\mathcal{J}_{\sigma}^3$ dominates all terms  $\mathcal{J}_{\sigma}^1$,  $\mathcal{J}_{\sigma}^2$, $\mathcal{J}_{\sigma}^4$ for $\sigma$ small enough.  Hence,  for $0<\sigma\ll1$, there holds
$$
\left(\wt L^\sigma_{B_{\sigma^2}}-\max_{\ol\Om}a_{T}+\sigma^2+ \sqrt{\sigma}\right)[\ol\phi_\epsilon]\geq0\quad\text{in}\quad \R\times \ol B_{\sigma^2}.
$$  
By Theorem  \ref{recal}, we have
$$
\wt\la_{p}(a,-\wt L^\sigma_{B_{\sigma^2}})=\wt\la_{p}'(a,-\wt L^\sigma_{B_{\sigma^2}})\leq -\max_{\ol\O}a_T+\sigma^2+ \sqrt{\sigma}.
$$ 
Proposition 6.1(2)\cite{ShenVo1} yields $\wt\la_{p}(a,-\wt L^\sigma_{\Omega})\leq \wt\la_{p}(a,-\wt L^\sigma_{B_{\sigma^2}})$ and thus
\begin{equation}\label{14.7.1}
\wt\la_{p}(a,-\wt L^\sigma_{\Omega})\leq -\max_{\ol\O}a_T+\sigma^2+ \sqrt{\sigma}.
\end{equation}
Let $\wt a^\sigma(t,x)=a(t,x)+D-D\int_{\frac{\Om-x }{\sigma}}J(z)dz$, obviously
\begin{equation}\label{14.7.2}
\lim_{\sigma\to 0}\sup_{t\in[0,T]}\|\wt a^\sigma(t,x)-a(t,x)\|_\infty=0,
\end{equation} 
and we derive, by Proposition 6.1(3),\cite{ShenVo1} that
\begin{equation}\label{14.7.3}
|\wt\la_{p}(\wt a^\sigma,-\wt L^\sigma_{\Omega})-\wt\la_{p}(a,-\wt L^\sigma_{\Omega})|\leq\sup_{t\in[0,T]}\|\wt a^\sigma(t,x)-a(t,x)\|_\infty,
\end{equation}
where   $\wt\la_{p}(\wt a^\sigma,-\wt L^\sigma_{\Omega})$, for $\sigma$ small enough,  is the principal eigenvalue of the operator
$$
\wt L^{\sigma}_\Omega[\psi]=-\psi_{t}(t,x)+D\int_{{\O}}J_{\si}(x-y)(\psi(t,y)-\psi(t,x)) dy+ a(t,x)\psi(t,x).
$$
Combining (\ref{14.7.1}), (\ref{14.7.2}), (\ref{14.7.3}), we pass to the limit as $\sigma\to0$ and get the desired inequality
 $$
\limsup_{\si\to0^{+}}\la_{1}(-L_{\si})= \limsup_{\sigma\to0}\wt\la_{p}(\wt a^\sigma,-\wt L^\sigma_{\Omega})\leq -\max_{\ol\O}a_T, 
 $$
which proves (\ref{17.5.1-03-12-17}).

\end{proof}


\section{{\bf Maximum principle}}

In this section, we prove maximum principle .

\begin{proof}[\bf Proof of Theorem \ref{thm-mp-1}]
Assume, without loss of generality, that $\si=1$. Let $\la_{1}=\la_{1}(-L)$ for simplicity and $\phi\in\XX^{++}$ be a principal eigenfunction of $-L$ associated to $\la_{1}$.

(1)$\implies$(2). Note that $L[\phi]=-\la_{1}\phi$. Since $\la_{1}\geq0$, (2) follows with $\vp=\phi$. 

(2)$\implies$(1). Let $\vp\in\XX^{++}$ be such that $L[\vp](t,x)\leq0$ for all $(t,x)\in\R\times\ol{\Om}$. Set $w=\frac{\vp}{\phi}$. Using the equation $L[\phi]+\la_{1}\phi=0$, it is easy to find that
\begin{equation*}
\begin{split}
L[\vp](t,x)&=L[w\phi](t,x)\\
&=-w_{t}(t,x)\phi(t,x)+D\int_{\Om}J\left(x-y\right)[w(t,y)-w(t,x)]\phi(t,y)dy-\la_{1}w(t,x)\phi(t,x).
\end{split}
\end{equation*}

Since $\min_{[0,T]\times\ol{\Om}}\vp>0$,  we have $\min_{[0,T]\times\ol{\Om}}w>0$. Let $(t_{0},x_{0})\in[0,T]\times\ol{\Om}$ be such that $w(t_{0},x_{0})=\min_{[0,T]\times\ol{\Om}}w$. We find that
\begin{equation*}
-w_{t}(t_{0},x_{0})\phi(t_{0},x_{0})=0\quad\text{and}\quad D\int_{\Om}J\left(x_{0}-y\right)[w(t_{0},y)-w(t_{0},x_{0})]\phi(t_{0},y)dy\geq0. 
\end{equation*} 
As $L[\vp](t_{0},x_{0})\leq0$, we find
$$
\la_{1}w(t_{0},x_{0})\phi(t_{0},x_{0})\geq-w_{t}(t_{0},x_{0})\phi(t_{0},x_{0})+D\int_{\Om}J\left(x_{0}-y\right)[w(t_{0},y)-w(t_{0},x_{0})]\phi(t_{0},y)dy\geq0.
$$
Since both $w(t_{0},x_{0})$ and $\phi(t_{0},x_{0})$ are positive, we conclude that $\la_{1}\geq0$.

(1)$\implies$(3). Let $\vp\in\XX_{\Om}$ be such that $L[\vp](t,x)<0$ for all $(t,x)\in\R\times\ol{\Om}$. Set $w=\frac{\vp}{\phi}$. Using the equation $L[\phi]+\la_{1}\phi=0$, it is easy to find that
\begin{equation*}
\begin{split}
L[\vp](t,x)&=L[w\phi](t,x)\\
&=-w_{t}(t,x)\phi(t,x)+D\int_{\Om}J\left(x-y\right)[w(t,y)-w(t,x)]\phi(t,y)dy-\la_{p}w(t,x)\phi(t,x).
\end{split}
\end{equation*}

Now, let us assume for contradiction that $\min_{[0,T]\times\ol{\Om}}\vp\leq0$. Then, $\min_{[0,T]\times\ol{\Om}}w\leq0$, and hence, there is $(t_{0},x_{0})\in[0,T]\times\ol{\Om}$ such that $w(t_{0},x_{0})=\min_{[0,T]\times\ol{\Om}}w$. It follows that 
\begin{equation*}
\begin{split}
-w_{t}(t_{0},x_{0})\phi(t_{0},x_{0})&=0,\\ 
D\int_{\Om}J\left(x_{0}-y\right)[w(t_{0},y)-w(t_{0},x_{0})]\phi(t_{0},y)dy&\geq0, \\
-\la_{1}w(t_{0},x_{0})\phi(t_{0},x_{0})&\geq0.
\end{split}
\end{equation*} 
Thus, $L_{\Om}[\vp](t_{0},x_{0})\geq0$, which leads to a contradiction.

(3)$\implies$(1). For contradiction, let us assume $\la_{1}<0$. Let $\Om_{0}\stst\Om$. The size of $\Om_{0}$ is to be specified. Let $\eta:\ol{\Om}\to[0,1]$ be a continuous function satisfying $\eta=1$ on $\Om_{0}$ and $\eta=0$ on $\partial\Om$. By the equality $L[\phi]+\la_{1}\phi=0$, we calculate
$$
L[\eta\phi](t,x)=D\int_{\Om}J\left(x-y\right)[\eta(y)-\eta(x)]\phi(t,y)dy-\la_{p}\eta(x)\phi(t,x).
$$
We consider three cases.
\begin{enumerate}
\item[(i)] If $x\in\ol{\Om}_{0}$, then 
$$
L[\eta\phi](t,x)=D\int_{\Om\sm\ol{\Om}_{0}}J\left(x-y\right)[\eta(y)-1]\phi(t,y)dy-\la_{p}\phi(t,x). 
$$
Since $\min_{\R\times[0,T]}\phi>0$, we deduce $L_{\Om}[\eta\phi](t,x)>0$ by simply choosing $\Om_{0}$ to be sufficiently close to $\Om$ so that the Lebesgue measure of $\Om\sm\ol{\Om}_{0}$ is sufficiently. 

\item[(ii)] If $x\in\ol{\Om}\sm\ol{\Om}_{0}$ and $\eta(x)\geq\frac{1}{2}$, then
$$
L[\eta\phi](t,x)\geq D\int_{\{y\in\Om:\eta(y)\leq\eta(x)\}}J\left(x-y\right)[\eta(y)-\eta(x)]\phi(t,y)dy-\frac{\la_{1}}{2}\phi(t,x).
$$
Note that in this case, there holds $\{y\in\Om:\eta(y)\leq\eta(x)\}\subset\Om\sm\ol{\Om}_{0}$. Therefore, choosing $\Om_{0}$ to be sufficiently close to $\Om$ ensures $L_{\Om}[\eta\phi](t,x)>0$.

\item[(iii)]  If $x\in\ol{\Om}\sm\ol{\Om}_{0}$ and $\eta(x)<\frac{1}{2}$, then
\begin{equation*}
\begin{split}
L[\eta\phi](t,x)&\geq D\int_{\ol{\Om}_{0}}J\left(x-y\right)[\eta(y)-\eta(x)]\phi(t,y)dy+D\int_{\Om\sm\ol{\Om}_{0}}J\left(x-y\right)[\eta(y)-\eta(x)]\phi(t,y)dy\\
&\geq\frac{D}{2}\int_{\ol{\Om}_{0}}J\left(x-y\right)\phi(t,y)dy+D\int_{\Om\sm\ol{\Om}_{0}}J\left(x-y\right)[\eta(y)-\eta(x)]\phi(t,y)dy.
\end{split}
\end{equation*}
Since $J(0)>0$, the integral $\frac{D}{2}\int_{\ol{\Om}_{0}}J\left(x-y\right)\phi(t,y)dy$ is uniformly positive for all $\Om_{0}$ sufficiently close to $\Om$. Choosing $\Om_{0}$ to be sufficiently close to $\Om$, we can make sure the term $D\int_{\Om\sm\ol{\Om}_{0}}J\left(x-y\right)[\eta(y)-\eta(x)]\phi(t,y)dy$ is sufficiently small, and hence, $L[\eta\phi](t,x)>0$.
\end{enumerate}

In conclusion, we can choose $\Om_{0}$ to be sufficiently close to $\Om$ to guarantee $L[\eta\phi](t,x)>0$ for all $(t,x)\in\R\times\ol{\Om}$. 

Since $\eta\phi\in\XX_{\Om}$, we apply (3) to $-\eta\phi$ to conclude that $-\eta\phi$ is strictly positive on $\R\times\ol{\Om}$, which leads to a contradiction.
\end{proof}

\begin{proof}[\bf Proof of Theorem \ref{thm-mp-2}]
Assume, without loss of generality, that $\si=1$. Let $\la_{1}=\la_{1}(-L)$ for simplicity and $\phi\in\XX^{++}$ be a principal eigenfunction of $-L$ associated to $\la_{p}$.

(1)$\implies$(2). Note that $L[\phi]=-\la_{1}\phi$. Therefore, if $\la_{1}>0$, (2) follows. 

(2)$\implies$(1). Let $\vp\in\XX^{+}$ be such that $L[\vp](t,x)<0$ for all $(t,x)\in\R\times\ol{\Om}$. Set $w=\frac{\vp}{\phi}$. Using the equation $L[\phi]+\la_{1}\phi=0$, it is easy to find that
\begin{equation*}
\begin{split}
L[\vp](t,x)&=L[w\phi](t,x)\\
&=-w_{t}(t,x)\phi(t,x)+D\int_{\Om}J\left(x-y\right)[w(t,y)-w(t,x)]\phi(t,y)dy-\la_{1}w(t,x)\phi(t,x).
\end{split}
\end{equation*}

We claim that $\inf_{[0,T]\times\ol{\Om}}\vp>0$. In fact, if $\min_{[0,T]\times\ol{\Om}}\vp=0$, then $\min_{[0,T]\times\ol{\Om}}w=0$. Let $(t_{0},x_{0})\in[0,T]\times\ol{\Om}$ be such that $w(t_{0},x_{0})=\min_{[0,T]\times\ol{\Om}}w$. We find that 
\begin{equation*}
\begin{split}
-w_{t}(t_{0},x_{0})\phi(t_{0},x_{0})&=0,\\ 
D\int_{\Om}J\left(x_{0}-y\right)[w(t_{0},y)-w(t_{0},x_{0})]\phi(t_{0},y)dy&\geq0, \\
-\la_{1}w(t_{0},x_{0})\phi(t_{0},x_{0})&=0.
\end{split}
\end{equation*}
It follows that $L_{\Om}[\vp](t_{0},x_{0})\geq0$, which leads to a contradiction.

Now, $\inf_{[0,T]\times\ol{\Om}}\vp>0$ implies that $\min_{[0,T]\times\ol{\Om}}w>0$. It then follows that 
$$
-w_{t}(t_{0},x_{0})\phi(t_{0},x_{0})=0\quad\text{and}\quad D\int_{\Om}J\left(x_{0}-y\right)[w(t_{0},y)-w(t_{0},x_{0})]\phi(t_{0},y)dy\geq0.
$$
Since $L[\vp](t_{0},x_{0})<0$, we find
$$
\la_{1}w(t_{0},x_{0})\phi(t_{0},x_{0})>-w_{t}(t_{0},x_{0})\phi(t_{0},x_{0})+D\int_{\Om}J\left(x_{0}-y\right)[w(t_{0},y)-w(t_{0},x_{0})]\phi(t_{0},y)dy\geq0.
$$
Since both $w(t_{0},x_{0})$ and $\phi(t_{0},x_{0})$ are positive, we conclude that $\la_{1}>0$.
\end{proof}

\section{Appendix}

The following is not the main result of this paper but it may be of the interest of a number of readers.
\begin{theorem}\label{cor1}
Suppose $J$ is radially symmetric. Let $k>2$.  Suppose $N\geq2$ and $a(t,x)=a(x)$.  Then, 
$$
\lim_{\si\to0^{+}}\lambda_1(-L_{\si,k})=-\frac{1}{|\Omega|}\int_{\O}a(x)dx.
$$ 
\end{theorem}

 Although the limt is only obtained for the case $a$  independent of $t$, but we must use a deep compactness result of Ponce's \cite{Ponce1} for $N\geq2$. The case $N=1$ is more involved due to the lack of compactness result, we believe that this is a hard question and leave it for a future work. We also guess that when $a$ depends periodically on $t$,  $k>2$, the limit  will be
$$
\lim_{\si\to0^{+}}\lambda_1(-L_{\si,k})=-\frac{1}{|\Omega|T}\int_0^T\int_{\O}a(t,x)dxdt.
$$ 
Our technique can be applied immediately as soon as such a compactness result for time dependent operator is obtained.

\begin{proof}[Proof of Theorem \ref{cor1}]
 For fixed $k>2$. We write $L_{\si,k}$ as $L_{\si}$, and $\la_{1}(-L_{\si,k})$ as $\la_{\si}$. Since $a(t,x)=a(x)$, $L_\si$ is independent of $t$. Let $\phi_\si$ be a principal eigenfunction so that
\begin{equation}
\left\lbrace
\begin{array}{ll}
\displaystyle\frac{D}{\si^k}\int_{\Om}J\left(\frac{x-y}{\si}\right)\frac{\phi_\si(y)-\phi_\si(x)}{\si^N}dy+ a(x)\phi_\si(x)+\lambda_\si\phi_\si(x)=0, &x\in\ol{\O},\\
\|\phi_\si\|_{L^\infty(\O)}=1.\\
\end{array}
\right.\label{eigen_e1}
\end{equation}
By Theorem \ref{characterization}, we can take $(-\displaystyle\max_{\ol{\Om}}a,1)$ and $(-\displaystyle\min_{\ol{\Om}}a,1)$ to be the test eigenpairs for $\lambda_\si$. As a result, $\{\lambda_\si\}_{\si}$ is bounded. Let $\{\la_{\si_{n}}\}_{n}\subset\{\la_{\si}\}_{\si}$ be an arbitrary sequence. Due to the boundedness of $\{\la_{\si}\}_{\si}$, up to a subsequence, we may assume without loss of generality that $\la_{\si_{n}}\to\lambda_0$ as $n\to\infty$. We show that the limit  $\lambda_0=-\frac{1}{|\Omega|}\int_{\O}a(x)dx$ holds, and thus, the result follows.

Multiplying the equation in \eqref{eigen_e1} by $\phi_\si$ and integrating over  $\O$, we have
$$
\frac{D}{\si^{k}}\int_{ \Omega}\int_{\O}J\left( \frac{x-y}{\si}\right)\frac{\phi_\si(y)-\phi_\si(x)}{\si^{N}}\phi_\si(t,x)dydx+\int_{\Omega}\left[a(x)+\lambda_\si\right]\phi_\si^2 dx =0.
$$
Calculations as in the proof of Lemma \ref{lem-poincare-inequality} lead to
\begin{equation*}
\frac{D}{2\si^{N+k}}\int_{ \Omega}\int_{\O}J\left(\frac{x-y}{\si}\right)\left[\phi_\si(y)-\phi_\si(x)\right]^2dydx+\int_{ \Omega}\left[a(x)+\lambda_\si\right]\phi_\si^2 dx=0.
\end{equation*}
As a result, there is $M_{1}>0$ such that
\begin{equation*}
\int_{\O\times\Omega}J\left( \frac{x-y}{\si}\right)\frac{1}{\si^{N+2}}\left(\phi_\si(y)-\phi_\si(x)\right)^2dydx\leq M_1\si^{k-2}.
\end{equation*}

Let $\rho(x)=\frac{J(x)|x|^2}{\int_{\R^N}J(x)|x|^2dx}$ for $x\in\R^{N}$ and $\rho_\si=\frac{1}{\si^{N}}\rho(\frac{\cdot}{\si})$. Then,
\begin{equation*}
\begin{cases}
\rho_\si(x)\geq0, & x\in\R^N,\\
\displaystyle\int_{\R^N}\rho_\si(x)dx=1, &\forall\si>0,\\
\displaystyle\lim_{\si\to0}\int_{|x|\geq\delta}\rho_\si(x)dx=0, &\forall\delta>0,
\end{cases}
\end{equation*}
and
\begin{equation}\label{21.1.2}
\int_{\Om}\int_{\Omega}\rho_\si(x-y)\frac{|\phi_\si(y)-\phi_\si(x)|^2}{|x-y|^{2}}dydx\leq M_1\si^{k-2} .
\end{equation}
We apply  \cite[Theorem 1.2]{Ponce1} to conclude that $\{\phi_\si\}_{\si}$ is relatively compact in $L^2(\O)$ and there exists $\phi\in W^{1,2}(\O)$ such that up to a subsequence $\phi_\si\to\phi$ in $L^2(\O)$. Moreover, by letting $\si\to0$ in (\ref{21.1.2}) we obtain again by  \cite[Theorem 1.2]{Ponce1} that $\int_{\O}|\nabla \phi|^2dx=0$, which implies $\phi$ must be constant. By the normalization of $\phi_\si$, we get $\phi\equiv1$. 

On the other hand, integrating (\ref{eigen_e1}) over $\O$, one has $\int_{\O}[a(x)+\lambda_\si]\phi_\si dx=0$. Since $\phi_\si\to1$ in $L^2(\O)$, $a(x)$ and $\{\lambda_\si\}_\si$ is bounded, we can pass to the limit to obtain $\lambda_0=-\frac{1}{|\O|}\int_{\O}a(x)dx$.
\end{proof}





\end{document}